\documentclass[reqno]{amsart}

\usepackage{amsmath}
\usepackage{amsfonts}
\usepackage{amsthm}
\usepackage{amssymb}
\usepackage{graphicx}
\usepackage[rightcaption]{sidecap}
\usepackage{bm}
\usepackage{dsfont}
\usepackage{color}
\usepackage{hyperref}
\usepackage{enumerate}
\usepackage{comment}
\usepackage[font=footnotesize]{caption}
\usepackage[all]{xy}
\usepackage{xy}
\usepackage{tikz}
\usepackage{paralist}
\allowdisplaybreaks
\usepackage{pgfplots}
\pgfplotsset{%
   every tick label/.append style = {font=\tiny},
   every axis label/.append style = {font=\scriptsize}
}

\usepackage{a4wide}

\numberwithin{equation}{section}
\theoremstyle{plain}
\newtheorem{thm}{Theorem}[section]
\newtheorem{thm*}{Theorem}

\newtheorem{lemma}[thm]{Lemma}

\theoremstyle{definition}

\newtheorem{ex}[thm]{Example}
\newtheorem{rmk}[thm]{Remark}
\newcommand{\ric}{\mathrm{Ric}^{g,\sigma,U}_k}

\newcommand{\R}{\mathbb{R}}

\newcommand{\M}{\mathcal{M}}

\newcommand{\speed}{\dot{\gamma}}

\title{Electromagnetic curvature via Jacobi-Maupertuis and beyond}
\author{V.Assenza \& G.Testolina}

\address{Instituto de Matem\'atica Pura e Aplicada}
\email{valerio.assenza@impa.br}

\address{Ruhr-Universität Bochum, Fakultät für Mathematik}
\email{giorgia.testolina@rub.de}

\begin{document}

\maketitle

\begin{abstract}
In the setting of electromagnetic systems, we propose a new definition of electromagnetic Ricci curvature, naturally derived via the classical Jacobi--Maupertuis reparametrization from magnetic Ricci curvature \cite{assenza,assenza2}. On closed manifolds, we show that if the magnetic force is nowhere vanishing and the potential is sufficiently small in the $C^2$ norm, then this Ricci curvature is positive for energies close to the maximum value of the potential $e_0$. As a main application, under these assumptions, we extend the existence of contractible closed orbits at energy levels near $e_0$ from \emph{almost every} to \emph{everywhere}.
\end{abstract}

\section{Introduction and results}
\subsection{Electromagnetic dynamics}
Let $M$ be an $n$-dimensional closed manifold with $n\geq2$, and consider on it a Riemannian metric $g$, a closed $2$-form $\sigma$, and a smooth function $U$. In this paper we are interested in studying some aspects of the dynamics arising from the second-order ordinary differential equation
\begin{equation}\label{Newtonequation}
    {\nabla^g}_{\speed} \speed - Y^{g,\sigma}\speed + \mathrm{grad}_g U = 0.
\end{equation}
Here $\nabla^g$ and $\mathrm{grad}_gU$ denote respectively the Levi-Civita connection and the gradient of $U$ induced by the metric $g$ while $Y^{g,\sigma}$ is the Lorentz endomorphism associated with $g$ and $\sigma$, defined by the identity
\begin{equation}\label{identitycondition}
    g(Y^{g, \sigma}v,w) = \sigma(v,w), \ \ \forall v,w \in TM.
\end{equation}
A solution $\gamma:\R\to M$ of equation \eqref{Newtonequation} is called an \textit{electromagnetic geodesic} or a $(g,\sigma,U)$-\textit{geodesic}. In fact, equation \eqref{Newtonequation} describes the motion, in the Riemannian structure induced by $g$, of a particle of unit mass and charge under the influence of a static magnetic force $\sigma$ and a static electric potential $U$. Observe that if the potential is trivial, i.e. $U=0$, the dynamics is purely magnetic. Conversely, if $\sigma=0$, the dynamics corresponds to that of a particle moving under the influence of a scalar potential only. If both $U=0$ and $\sigma=0$, equation \eqref{Newtonequation} reduces to the classical geodesic equation associated with the metric $g$.

It is a classical fact that the mechanical energy
\[
E^{g,U}(p,v) = \frac{1}{2}g(v,v) + U(p)
\]
is constant along electromagnetic geodesics. However, if either $\sigma$ or $U$ is nontrivial, equation \eqref{Newtonequation} is nonhomogeneous, so that the electromagnetic dynamics may drastically change while ranging across different energy levels ${(E^{g,U})}^{-1}(k)=E_k$ for values $k > U_*$, where $U_*$ denotes the minimal value of $U$ on $M$. In this regard, we now introduce two critical energy values that mark important changes in the dynamics. 

Define the maximal value of the potential
\[
e_0 = e_0(U) := \max_{M} U  ;
\]
and the \textit{Ma\~n\'e critical value} $c$, given by
\begin{equation*}
c = c(g,\sigma,U) := \begin{cases}\displaystyle \inf_{\theta\in \Omega(\widetilde{M}), \ \mathrm{d}\theta = p^* \sigma}  \sup_{\widetilde{M}} \left\{ \frac{1}{2} g(\theta,\theta) + \widetilde{U} \right\},   \ \ \mathrm{if} \ p^*\sigma \ \textrm{is exact},  \\ +\infty, \ \ \ \ \ \ \ \ \ \ \ \ \ \ \ \ \ \ \ \ \ \ \ \ \ \ \ \ \ \  \ \ \ \ \ \mathrm{otherwise}.\end{cases}
\end{equation*}
Here, $p:\widetilde{M} \to M$ denotes the universal covering, $\widetilde{U}=p^*U$ and, with a slight abuse of notation, $g$ stands for the dual metric induced by the lift of $g$ to $\widetilde{M}$. By these definitions, it follows that
\[U_*\leq e_0 \leq c\leq +\infty.\]
It is clear that $c$ is finite if and only if the lift of $\sigma$ to $\widetilde{M}$ admits a bounded primitive. In general, the values $e_0$ and $c$ may coincide. However, as pointed out in~\cite[Remark~1.2]{assellemazzucchelli}, by fixing $g$ and $\sigma\neq \emptyset$, there exists a $C^1$–neighborhood of the identically zero function such that, for every potential in this neighborhood, the inequality between $e_0$ and $c$ is strict.
The centrality of these values becomes evident when addressing the problem of the existence of a closed $(g,\sigma,U)$-geodesic with prescribed energy $k$, i.e.\ a periodic solution of \eqref{Newtonequation} lying on a given energy level $E_k$ with $k >U_*$. We refer the reader to \cite{abbo2013} and \cite{benedettiLectures} for excellent preliminary introductions to this problem. In general, this question is approached within a variational framework, a significant part of which will be discussed later in this work. In fact, closed $(g,\sigma,U)$-geodesics with energy $k$ correspond to the zeros of a suitable $1$-form $\alpha_k$ defined on the space of loops with free period (see for instance \cite[Section~2]{lus-fet}). For energies $k>c$, it was proved in \cite{Contreras1}\footnote{The argument in \cite{Contreras1} applies to the case where the magnetic form $\sigma$ is exact. A generalization of this argument to the weakly exact case can be found in \cite{Merry} and \cite{benedettiLectures}.} that $\alpha_k$ satisfies a compactness condition in the sense of Palais-Smale. Consequently, the existence of a closed $(g,\sigma,U)$-geodesic follows from classic variational arguments. 

For energies below $c$, however, $\alpha_k$ generally fails to satisfy the Palais–Smale conditions, making the variational analysis considerably more delicate. In the general setting, the first existence result, appeared in \cite{Contreras2} and subsequently refined in \cite{Merry,lus-fet}, asserts the existence of a contractible closed $(g,\sigma,U)$-geodesic at almost every prescribed energy $k\in(U_*,c)$. The underlying idea is to adapt a Struwe-type monotonicity method to a minimax scheme, parametrized by $k$, obtained by studying the behavior of $\alpha_k$ in a neighborhood of constant loops. It is conjectured that the existence of a contractible (or not) closed $(g,\sigma,U)$-geodesic extends for every energy level in $(U_*,c)$.

Under additional assumptions, significant progress has been made toward this conjecture. For instance, when $M$ is a closed surface and the magnetic form $\sigma$ is exact, existence had been established for every $k$ strictly contained in the range between $e_0$ and $c$, in \cite{assellemazzucchelli}, extending previous results obtained in \cite{contrerasmacarinipaternain} and \cite{Taimanov}. In the purely magnetic case, i.e. when $U=0$, under the nowhere vanishing assumption on $\sigma$, the existence of a contractible closed $(g,\sigma)$-geodesic is guaranteed for every $k$ sufficiently close to $e_0$ from above. With the development of new localization techniques, this result has been recently extended to the case where $\sigma$ is \emph{spherically rational}\footnote{Spherically rational means that the map $I^{\sigma}:\pi_2(M)\to \R$ obtained by integration of $\sigma$ has discrete image.}, as shown in \cite{Assenza3}. We also point the reader to 
\cite{ginzburggurel} and \cite{usher}, where with a different array of techniques, analogous results have been obtained under the stronger assumption of $\sigma $ being symplectic. Although this topic will not be treated in the present work, we refer the reader to \cite{assellemazzucchelli,AMP2,AMMP1} for results concerning the multiplicity of closed $(g,\sigma)$-geodesics with prescribed energy.

The main result in this paper asserts that if the magnetic form is nowhere vanishing and the potential is sufficiently small in the $C^2$-norm, then the existence of a contractible closed $(g,\sigma,U)$-geodesic extends to every energy sufficiently close to $e_0$ from above. In particular, this theorem provides a perturbative generalization of \cite[Theorem~A1]{assenza}.

\begin{thm}\label{TheroemA1}
Let $\sigma$ be a nowhere vanishing magnetic form on a closed manifold $M$. 
Then, for every metric $g$, there exists a $C^2$-neighborhood $\mathcal{V}$ of
the identically zero function such that, if $U \in \mathcal{V}$, there exists
$\nu_{\mathrm{Ric}} \in (e_0, c]$, depending on $g$, $\sigma$, and $U$, with
the property that, for every $k \in (e_0, \nu_{\mathrm{Ric}}]$, the energy
level $E_k$ carries a contractible closed $(g,\sigma,U)$-geodesic.

\end{thm}

The proof scheme of Theorem~\ref{TheroemA1}, as well as the nature of the value $\nu_{\mathrm{Ric}}$, relies on the notion of electromagnetic curvature, which we detail the construction in the next section. A central role in this construction is played by the so-called 
\emph{Jacobi--Maupertuis principle}\footnote{A classical reference for the Jacobi--Maupertuis principle is~\cite[Chapter 9]{arnold}}, which allows, for energies above $e_0$ and up to a time reparametrization, to conjugate the electromagnetic dynamics to a purely magnetic one. More precisely, if $k>e_0$ and $\gamma$ is a solution of \eqref{Newtonequation} lying in $E_k$, consider the curve $\tilde{\gamma} = \gamma \circ t$ obtained through the reparametrization $t:\R \to \R$ defined by 
\begin{equation}\label{reparametrization}
t(s)=\int_0^s \frac{1}{2\big(k-U(\gamma(\tau))\big)}\,\mathrm{d}\tau.
\end{equation}
Denoting by ``$\, '\,$'' the derivative with respect to $s$, one easily checks that $\tilde{\gamma}$ satisfies
\begin{equation}\label{JMequation}
    \begin{cases} \displaystyle
        {\nabla^{g_k}}_{\tilde{\gamma}'} \tilde{\gamma}' - Y^{g_k,\sigma}\tilde{\gamma}' = 0,\\[0.5em]
        g_k(\tilde{\gamma}',\tilde{\gamma}') = 1,
    \end{cases}
\end{equation}
where $g_k$ denotes the metric conformal to $g$ given by
\begin{equation}\label{conformalmtric}
    g_k = 2(k-U)\,g.
\end{equation}
In other words, $\tilde{\gamma}$ is a $(g_k,\sigma)$-geodesic with $E^{g_k}(\tilde{\gamma},  \tilde{\gamma}^\prime) = \frac12$. We write $c_k = c(g_k, \sigma)$. A key aspect is that if $k < c$, then $\frac12 < c_k$. This allows us to approach the problem of finding closed $(g,\sigma,U)$-geodesics with energy $k \in (e_0, c)$ within the same framework used to find $(g_k, \sigma)$-geodesics with energies below $c_k$. The Jacobi–Maupertuis principle was previously adopted in~\cite{Benci} for the problem of finding closed trajectories in the case of natural Hamiltonians, that is, when $\sigma = 0$.
 
\subsection{Electromagnetic curvature}
In the context of magnetic systems, a notion of magnetic curvature functions has been recently introduced in full generality in \cite{assenza,assenza2}. Similar constructions have appeared previously in~\cite{BahriTaimanov} and~\cite{Wojtkowski}. We briefly retrace, in a reduced form, the definition of the corresponding magnetic Ricci curvature. For a detailed treatment, we refer the reader to the aforementioned references. 

Given a metric $g$, we write $S^gM$ the bundle of $g$-unitary vectors and denote by $\mathrm{Ric}^{g}:S^gM\to \R$ the Ricci curvature function related to $g$. For every closed 2-form $\sigma$, and for a fixed $v\in S^gM$, we consider the following operators:
\begin{align}
   &(\nabla Y^{g,\sigma})_v: w \mapsto ({\nabla^g}_w Y^{g, \sigma}) v, \\
   &(\widetilde{Y}^{g,\sigma})_v : w \mapsto \frac{3}{4}g(w,Y^{g, \sigma}v)Y^{g, \sigma}v - \frac{1}{4}Y^{g, \sigma}Y^{g, \sigma}w. \label{defYtilde}
\end{align}
The magnetic Ricci curvature is the function $\mathrm{Ric}^{g,\sigma}:S^gM\to \R$ defined as
\begin{equation}\label{riccimagnetica}
\mathrm{Ric}^{g,\sigma}(v) := \mathrm{Ric}^{g}(v)
- \mathrm{trace}\left( (\nabla Y^{g,\sigma} - \widetilde{Y}^{g,\sigma})_v \right)
,
\end{equation}
Starting from this definition, it is natural to extend the Ricci curvature to the electromagnetic case on energy levels $E_k$ with $k>e_0$, by means of the Jacobi-Maupertuis reparametrization. In detail, given $g$, $\sigma$, and $U$, and for $k>e_0$, let $g_k$ be the conformal metric defined in \eqref{conformalmtric}. Denote by $P_k:E_k\to S^{g_k}M$ the diffeomorphism defined by
\[
P_k(v) = \frac{v}{\sqrt{g_k(v,v)}}.
\]
Define the \textit{Ricci electromagnetic curvature} at level $k$ as the function $\mathrm{Ric}^{g,\sigma,U}_k:E_k \to \R$ given by
\begin{equation}\label{electromagnetic}
\mathrm{Ric}^{g,\sigma,U}_k (v):=\mathrm{Ric}^{g_k,\sigma}\circ P_k(v).
\end{equation}
We give the explicit computation of the electromagnetic Ricci curvature in Lemma~\ref{em}, and discuss the special case of closed surfaces in Section~\ref{ricc_surface}. 

The next result is a generalization to the electromagnetic setting of \cite[Theorem~A]{assenza}. It asserts that any energy level between $e_0$ and $c$ at which the electromagnetic Ricci curvature is positive carries a contractible closed electromagnetic geodesic.
\begin{thm}\label{thm_existence}
Let $(g, \sigma, U)$ be an electromagnetic system on a closed manifold. 
If $k \in (e_0, c]$ is such that $\mathrm{Ric}^{g, \sigma, U}_k > 0$, then the energy level $E_k$ carries a contractible closed $(g, \sigma, U)$-geodesic.
\end{thm}
Thanks to the Jacobi–Maupertuis reparametrization, with minor adjustments, we can articulate the proof of Theorem~\ref{thm_existence} following a variational scheme analogous to the purely magnetic case. In fact, closed $(g_k,\sigma)$-geodesics correspond to the zeros of a one-form $\alpha_k$, where in our setting $k>e_0$ parametrizes the conformal metric $g_k$. For values of $k$ below or equal to $c$, by introducing an auxiliary parameter $\lambda$, we reproduce a Struwe-type argument based on the classical minimax geometry constructed around constant loops and adapted to perturbations $\alpha_{k,\lambda}$ of $\alpha_k$. In this way, we are able to construct Palais--Smale sequences for $\alpha_k$ consisting of zeros of $\alpha_{k,\lambda}$. Generally, Struwe’s technique does not guarantee any compactness condition for such sequences. However, under the assumption of positive curvature, a Bonnet--Myers type argument allows us to obtain a bound on the period of the zeros of $\alpha_{k,\lambda}$ in function of their Morse index. This technique, combined with classical index estimates for critical points arising from minimax geometry, restores the missing compactness condition and consequently ensures the existence of a zero of $\alpha_k$. This scheme was first employed in~\cite{BahriTaimanov} and was subsequently
revisited in various works dealing with variational problems arising from
different contexts (see, for instance, \cite{xin}). In the context of the
existence of closed magnetic geodesics, it was recently refined in
\cite{assenza} in a setting similar to the one adopted in the present work.

In view of Theorem \ref{thm_existence}, it is natural to investigate under which conditions $\mathrm{Ric}^{g,\sigma,U}$ is positive. In \cite[Lemma~6]{assenza}, it is proved that if $U=0$ and the magnetic form $\sigma$ is nowhere vanishing, then the magnetic Ricci curvature is positive and uniformly bounded away from zero for every value of $k$ close to $0$. By means of a perturbative argument, we show that this result continues to hold also in the case where the magnetic form is nowhere vanishing and the potential $U$ is sufficiently small with respect to the $C^2$-norm. In particular, we define a new critical value $\nu_\mathrm{Ric}=\nu_\mathrm{Ric}(g,\sigma,U)$ as
\[
\nu_\mathrm{Ric}:=\sup \{k>e_0 \mid \ \mathrm{Ric}_s^{g,\sigma,U}>0, \ \forall s\in (e_0,k) \} \in [e_0,+\infty).
\]
\begin{thm} \label{thm_positivity}
Let $\sigma$ be a nowhere vanishing magnetic form on a closed manifold. 
Then, for every metric $g$, there exists a $C^2$-neighborhood $\mathcal{V}$ of the identically zero function such that, for every potential $U \in \mathcal{V}$, we have $\nu_{\mathrm{Ric}} > e_0$.
\end{thm}
It is immediate to see that Theorems~\ref{thm_existence} and~\ref{thm_positivity} together imply Theorem~\ref{TheroemA1}. We further emphasize that the assumption on the the $C^2$-norm of the potential $U$ is essential in the statement of Theorem~\ref{thm_positivity}. Indeed, in Example~\ref{C1ex}, we construct a case where $\ric$ is not strictly positive for every $k>e_0$, the magnetic form $\sigma$ is nowhere vanishing, and $U$ is small in the $C^1$-norm but not in the $C^2$-norm. Finally, we would like to point out that, to the best of our knowledge, there are still no known non trivial examples in the literature of electromagnetic systems $(g,\sigma,U)$ for which $\nu_{\mathrm{Ric}} > c$.  \\

We conclude this introduction with a few general remarks. First, through the Jacobi–Maupertuis reparametrization, one can extend from the magnetic to the electromagnetic case other curvature functions as well, such as the sectional and scalar curvatures. We expect that these functions carry significant information about the electromagnetic dynamics above the energy level $e_0$. Building on the preliminary work in~\cite{assenza2} and ~\cite{Wojtkowski}, we are confident that the notion of electromagnetic curvature can also be derived, above and below $e_0$, by studying the linearized problem associated with equation~\eqref{Newtonequation}.

\section{Explicit formula for $\mathrm{Ric}^{g,\sigma,U}$ and proof of Theorem \ref{thm_positivity}}

We now derive an explicit expression for the electromagnetic Ricci curvature $\mathrm{Ric}^{g,\sigma,U}_k$ introduced in~\eqref{electromagnetic}.
The computation relies on the conformal relation between $g_k$ and $g$. To simplify the notation, if $g$ is a metric we shall use the symbol $\vert \cdot \vert_g$ to denote both the norm and the dual norm induced by $g$, letting the context clarify which one is meant. We write $\Vert \cdot \Vert_{\infty,g}$ the uniform norm induced by $g$. Observe that if $\sigma$ is a magnetic form, then $\Vert \sigma \Vert_{\infty,g}=\Vert Y^{g,\sigma} \Vert_{\infty,g}$. We denote $v = (p, v) \in E_k$ a point on the energy level, $\hat v = v / \lvert v \rvert_g$ its normalization with respect to $g$, and set $v_k = P_k(v)$. 
Moreover, $\mathrm{Hess}_g U$ and $\Delta_g U$ denote, respectively, the Hessian and the Laplace–Beltrami operator of $U$ induced by the metric $g$. Throughout this section, all quantities are evaluated at the base point $p$.

\begin{lemma}\label{em}
Let $(g,\sigma,U)$ be an electromagnetic system and let $k>e_0$. Then, for every $v \in E_k$, the electromagnetic curvature satisfies:
\begin{align}\label{eq:ricci_em}
\operatorname{Ric}_k^{g,\sigma,U}(v) 
&= \frac{\operatorname{Ric}^g(\hat v)}{2(k-U)}
+ \frac{n-2}{4(k-U)^2}\,(\mathrm{Hess}_gU)[\hat v,\hat v]
+ \frac{\Delta_g U}{4(k-U)^2} + \nonumber\\
&\quad
+ \frac{3(n-2)}{8(k-U)^3}\big(\mathrm{d}U(\hat v)\big)^2
+ \frac{4-n}{8(k-U)^3}\,|\mathrm{d}U|_g^2+ \\
&\quad
- \frac{\operatorname{trace}\big((\nabla Y^{g,\sigma})_{\hat v}\big)}{(2(k-U))^{3/2}}
+ \frac{(n-4)}{(2(k-U))^{5/2}}\, \mathrm{d}U \left(Y^{g,\sigma}(\hat v)\right)
+ \frac{\operatorname{trace}\left((\widetilde{Y}^{g,\sigma})_{\hat v}\right)}{4(k-U)^2} \ .\nonumber
\end{align}
\end{lemma}

\begin{proof}
By definition \eqref{electromagnetic}, we have
\begin{equation}\label{explicitcomp}
\operatorname{Ric}^{g,\sigma,U}_k(v) = \mathrm{Ric}^{g_k}\left(v_k\right) - \operatorname{trace}\left( (\nabla Y^{g_k,\sigma}-\widetilde{Y}^{g_k,\sigma})_{v_k}\right).
\end{equation}
We proceed by computing the two terms of \eqref{explicitcomp} separately. Denote by $f=\frac{1}{2} \ln\left(2(k-U)\right)$ so that $g_k= e^{2f}g$, and $v_k=e^{-f}\hat v$. Under the conformal change, the first term in \eqref{explicitcomp} reads as
\begin{equation}\label{term1}
\begin{split}
\operatorname{Ric}^{g_k}(v_k)
&= \operatorname{Ric}^g(v_k)
-(n-2)\Big((\operatorname{Hess}_gf)[v_k,v_k]-\big(df(v_k)\big)^2\Big)  +\\
&\qquad -\Big(\Delta_gf+(n-2)|df|_g^2\Big)\,g(v_k,v_k) \\
&= \frac{\operatorname{Ric}^g(\hat v)}{2(k-U)}
+ \frac{n-2}{4(k-U)^2}\,(\operatorname{Hess}_gU)[\hat v, \hat v]
+ \frac{\Delta_g U}{4(k-U)^2} +\\
&\qquad + \frac{3(n-2)}{8(k-U)^3}\big(\mathrm{d}U(\hat v)\big)^2
+ \frac{4-n}{8(k-U)^3}\,|\mathrm{d}U|_g^2,
\end{split}
\end{equation}
where in the second equality we used the identities
\[
df=-\frac{dU}{2(k-U)},\qquad
\operatorname{Hess}_gf=-\frac{\operatorname{Hess}_gU}{2(k-U)}-\frac{dU\otimes dU}{2(k-U)^2},\qquad
\Delta_gf=-\frac{\Delta_gU}{2(k-U)}-\frac{|dU|_g^2}{2(k-U)^2}.
\] 
To assist the reader with the next computations, let us recall that if $W$ and $Z$ are vector fields on $M$, and $A$ is an endomorphism of $TM$, then 
\begin{equation}\label{conformalidentity1}
{\nabla^{g_k}}_W Z = {\nabla^g}_W Z + \mathrm{d}f(W)Z + \mathrm{d}f(Z)W - g(W,Z)\,\mathrm{grad}_g f,
\end{equation}
\begin{equation}\label{conformalidentity2}
{\nabla^g}_W(AZ) = ({\nabla^g}_W A)Z + A({\nabla^g}_W Z).
\end{equation}
Concerning the second term, first observe that for every $v,w\in TM$, under the identities \eqref{conformalidentity1} and \eqref{conformalidentity2}, we have
\begin{equation}\label{computazione1}
\begin{split}
    \left( \nabla Y^{g_k,\sigma} \right)_{v_k}w &= ({\nabla^g}_w Y^{g_k,\sigma})v_k +\mathrm{d}f\left( Y^{g_k,\sigma}v_k \right)w - \mathrm{d}f\left( v_k\right)Y^{g_k,\sigma}w  +\\ 
    &\qquad - g(w,Y^{g_k,\sigma}v_k)\mathrm{grad}_gf +g(v_k,w)Y^{g_k,\sigma}\mathrm{grad}_gf \\
    &= e^{-3f} \Big[-2\mathrm{d}f(w)Y^{g,\sigma}\hat v +{(\nabla Y^{g,\sigma})_{\hat v}}w + \mathrm{d}f\left({Y^{g,\sigma}}\hat v\right)w\Big] + \\
    & \qquad -e^{-3f} \Big[ \mathrm{d}f \left( \hat v\right)Y^{g,\sigma}w+g\left( w,Y^{g,\sigma}\hat v\right)\mathrm{grad}_gf - g(\hat v,w)Y^{g,\sigma}\mathrm{grad}_g f \Big],
    \end{split}
\end{equation}
and similarly,
\begin{equation}\label{computazione2}
    \begin{split}
        \left( {\widetilde{Y}^{g_k,\sigma}}\right)_{v_k}w &= \frac{3}{4}g_k \left(w, Y^{g_k,\sigma}v_k \right)Y^{g_k,\sigma}v_k - \frac{1}{4}Y^{g_k,\sigma}Y^{g_k,\sigma}w \\
        &= e^{-4f}\Big[ \frac{3}{4}g \left(w, Y^{g,\sigma}\hat v \right)Y^{g,\sigma}\hat v - \frac{1}{4}Y^{g,\sigma}Y^{g,\sigma}w\Big] \\
        &= e^{-4f} \left( \widetilde{Y}^{g,\sigma}\right)_{\hat v}.
    \end{split}
\end{equation}

By completing $\hat v$ to a $g$-orthonormal basis and using \eqref{computazione1} and \eqref{computazione2}, together with the fact that $Y^{g,\sigma}$ is skew-adjoint, we conclude that
\begin{equation}\label{term2}
    \begin{split}
       \mathrm{trace}\left( (\nabla Y^{g_k,\sigma} -\widetilde{Y}^{g_k,\sigma})_{v_k}\right) &=e^{-3f} \Big[ -2\mathrm{d}f\left( Y^{g,\sigma}\hat v\right)+\mathrm{trace}\left( (\nabla Y^{g,\sigma})_{\hat v}\right)\Big]+ \\
       & \qquad + e^{-3f} \Big[n\mathrm{d}f\left(Y^{g,\sigma}\hat v\right)-2g\left(\mathrm{grad}_g f, Y^{g,\sigma}\hat v \right)\Big]+ \\ 
       &\qquad\qquad \qquad-e^{-4f}\mathrm{trace}\left( (\widetilde{Y}^{g,\sigma})_{\hat v}\right) \\
       &=e^{-3f} \Big[ (n-4)\mathrm{d}f\left( Y^{g,\sigma}\hat v \right)+\mathrm{trace}\left( (\nabla Y^{g,\sigma})_{\hat v}\right)\Big] + \\
       &\qquad \qquad \qquad \qquad -e^{-4f}\mathrm{trace}\left( (\widetilde{Y}^{g,\sigma})_{\hat v}\right)\\
       &= \frac{\operatorname{trace}\big((\nabla Y^{g,\sigma})_{\hat v}\big)}{(2(k-U))^{3/2}} -
 \frac{(n-4)}{(2(k-U))^{5/2}}\, \mathrm{d}U \!\left(Y^{g,\sigma}(\hat v)\right) + \\
 &\qquad \qquad\qquad - 
 \frac{\operatorname{trace}\left((\widetilde{Y}^{g,\sigma})_{\hat v}\right)}{4(k-U)^2}.
    \end{split}
\end{equation}
Substituting \eqref{term1} and \eqref{term2} into \eqref{explicitcomp} yields the desired expression, completing the proof.
\end{proof}

\begin{lemma}\label{lem:magnetic_lb}
Let $M$ be a closed manifold and $\sigma$ a nowhere vanishing magnetic form. Then, for every metric $g$, there exists a constant $C_g  > 0$ such that 
\[
\operatorname{trace}\big(\widetilde{Y}^{g,\sigma}\big)_v \geq C_g \Vert\sigma\Vert_{\infty,g}^2 , \quad \forall v\in S^gM.
\]
In particular, this implies that for every potential $U$ and every $k>e_0$, one has
\[
\operatorname{trace}\big(\widetilde{Y}^{g_k,\sigma}\big)_{v_k}
= \frac{\operatorname{trace}\big(\widetilde{Y}^{g,\sigma}\big)_{\hat v}}{4(k-U)^2}
\geq \frac{C_g \Vert\sigma\Vert_{\infty,g}^2}{4(k-U)^2}, \quad \forall v\in E_k.
\]

\begin{proof}
Let $v\in S^gM$ and complete $v$ to a $g$-orthonormal basis $\{v,e_2,\ldots,e_n\}$.  
By the definition \eqref{defYtilde} of $\widetilde{Y}^{g,\sigma}$, we obtain
\begin{align}
\operatorname{trace}\big((\widetilde{Y}^{g,\sigma})_v\big)
&= -\tfrac{1}{4}g\big(Y^{g,\sigma}Y^{g,\sigma}v,v \big)
+ \sum_{i\ge2} \left\{ \tfrac{3}{4}g(e_i,Y^{g,\sigma}v)^2
- \tfrac{1}{4}g\big(Y^{g,\sigma}Y^{g,\sigma}e_i,e_i\big)\right\} \nonumber\\
&= \vert Y^{g,\sigma}v \vert_g^2 +\frac{1}{4} \sum_{i\ge2} \vert Y^{g,\sigma}e_i \vert_g^2. \label{positiveterm}
\end{align}

If $\sigma$ is nowhere vanishing, then $Y^{g,\sigma}$ is nontrivial, and hence at least one of the terms in \eqref{positiveterm} is nonzero.  
By compactness of $M$, we can find a constant $C_g$ such that the function $v \mapsto \operatorname{trace}\big((\widetilde{Y}^{g,\sigma})_v\big)$ is bounded from below by $C_g \Vert\sigma\Vert_{\infty,g}^2$.  

Finally, from equation \eqref{term2} in Lemma~\ref{em}, we deduce that 
\[
\operatorname{trace}\big(\widetilde{Y}^{g_k,\sigma}\big)_{v_k}
= \frac{\operatorname{trace}\big(\widetilde{Y}^{g,\sigma}\big)_{\hat v}}{4(k-U)^2},
\]
which proves the second part of the statement and completes the proof.
\end{proof}
\end{lemma}

We are now ready to prove Theorem~\ref{thm_positivity}.

\begin{proof}[Proof of Theorem~\ref{thm_positivity}]
The argument employed here relies on estimating $\ric$ as $k$ approaches $e_0$ from above. By Lemma~\ref{lem:magnetic_lb}, the term involving the trace of $\widetilde{Y}^{g_k,\sigma}$ is strictly positive. This in fact implies the positivity of $\ric$, provided that the potential $U$ is sufficiently small in the $C^2$-norm and the energy level $k$ is sufficiently close to $e_0$.\\
Let $\varepsilon\in(0,1)$, whose precise size will be clarified later in the proof. Let $U$ be a potential satisfying $\Vert U\Vert_{C^2}<\varepsilon$. In particular, by the assumption on $U$, the gradient descent lemma \cite[Prop.~10.53]{Boumal} yields the following estimate:
\begin{equation}\label{gradescentlemma}
\vert \mathrm{d}U\vert^2_g \leq 2\varepsilon(e_0-U)\leq 2\varepsilon(k-U), \quad  \forall \, k>e_0.
\end{equation}
Moreover, there exists a constant $d_g>0$ such that
\begin{equation}\label{boundlaplacian}
    \vert \Delta_g U\vert \leq d_g \Vert U \Vert_{C^2}\leq d_g\,\varepsilon.
\end{equation}
Set 
\[ D_{g,\sigma} =\max \left\{1,\Vert Y^{g,\sigma}\Vert_{\infty,g},\Vert \mathrm{trace}\left( \nabla Y^{g,\sigma}\right)\Vert_{\infty,g},\Vert \mathrm{Ric}^g\Vert_{\infty,g} \right\}.\]
Write $\rho_k = 2(k-e_0)$. It follows that 
\begin{equation}\label{parte1}
    \begin{split}
        A_k(v) := \frac{\mathrm{Ric}^g(\hat v)}{2(k-U)}-\frac{\mathrm{trace} \left( (\nabla Y^{g,\sigma})_{\hat{v}}\right)}{\left( 2(k-U)\right)^{\frac{3}{2}}}
        &\geq -D_{g,\sigma}\left(\frac{1}{2(k-U)} + \frac{1}{\left( 2(k-U)\right)^{\frac{3}{2}}} \right)\\
        & \geq -\frac{D_{g,\sigma}}{4(k-U)^2}\left(\rho_k + \sqrt{\rho_k} \, \right).
    \end{split}
\end{equation}
On the other hand, by the assumption $\Vert U \Vert_{C^2}<\varepsilon<1$, the estimates \eqref{gradescentlemma} and \eqref{boundlaplacian}, and the fact that $D_{g,\sigma}\geq 1$, we also obtain
\begin{equation}\label{parte2}
\begin{split}
    B_k(v)&:= \frac{n-2}{4(k-U)^2}\,(\mathrm{Hess}_gU)[\hat v,\hat v]
+ \frac{\Delta_g U}{4(k-U)^2} + \frac{3(n-2)}{(2(k-U))^3}\big(\mathrm{d}U(\hat v)\big)^2+ \\
&\qquad
+ \frac{4-n}{(2(k-U))^3}\,|\mathrm{d}U|_g^2 + \frac{(n-4)}{(2(k-U))^{5/2}}\, \mathrm{d}U \!\left(Y^{g,\sigma}\hat v\right) \\
&\geq \frac{n-2}{4(k-U)^2}\,(\mathrm{Hess}_gU)[\hat v,\hat v]
+ \frac{\Delta_g U}{4(k-U)^2} - \lvert n -4 \rvert \Big( \frac{|\mathrm{d}U|_g^2}{(2(k-U))^3} + \frac{\mathrm{d}U \!\left(Y^{g,\sigma}\hat v\right)}{(2(k-U))^{5/2}} \Big) \\
& \geq -\frac{\varepsilon  \left(n-2 + d_g \right)}{4(k-U)^2} - \lvert n-4 \rvert \Big( \frac{ \varepsilon + \sqrt{\varepsilon }\,D_{g,\sigma}}{4(k-U)^2} \Big) \\
& \geq -\frac{\varepsilon  \, D_{g, \sigma}\left(n-2 + d_g \right)}{4(k-U)^2} - \frac{ \sqrt{\varepsilon } \, 2 D_{g, \sigma} \lvert n-4 \rvert } {4(k-U)^2} \,  \\
& \geq -\frac{\sqrt{\varepsilon}\, D_{g,\sigma}}{4(k-U)^2} \, G_{g, n},
\end{split}
\end{equation}
where we write 
\[
G_{g,n} :=  n+ 2 \lvert n-4 \rvert + d_g - 2.
\]
By combining \eqref{parte1}, \eqref{parte2}, and Lemma~\ref{lem:magnetic_lb}, we finally obtain
\begin{equation*}
    \begin{split}
        \ric(v) &= A_k(v)+B_k(v) + \frac{\operatorname{trace}\big(\widetilde{Y}^{g,\sigma}\big)_{\hat v}}{4(k-U)^2}\\
        &\geq  -\frac{D_{g,\sigma}}{4(k-U)^2}\left(\rho_k +\sqrt{\rho_k} \, \right)
        - \frac{\sqrt{\varepsilon}\, D_{g,\sigma}G_{g, n}}{4(k-U)^2}
        + \frac{C_g\Vert\sigma\Vert_{\infty,g}^2}{4(k-U)^2} \\
        &\geq \frac{D_{g,\sigma}}{4(k-U)^2}\left[-(\rho_k +\sqrt{\rho_k} \,) -\sqrt{\varepsilon}G_{g, n}+\frac{C_g\Vert\sigma\Vert_{\infty,g}^2}{D_{g,\sigma}}\right].
    \end{split}
\end{equation*}
Since $\rho_k \searrow 0$ when $k\to e_0$, if 
\[ \varepsilon \in \left(0, \frac{{C_g}^2 {\Vert\sigma\Vert_{\infty,g}^4}}{{{D_{g,\sigma}}^2} \, {G_{g,n}}^2} \right),\]
then $\ric$ is strictly positive for every $k$ sufficiently close to $e_0$. In particular, 
\[\nu_{\mathrm{Ricc}}(g,\sigma,U)= \sup \{k>e_0 \mid \ \mathrm{Ric}_s^{g,\sigma,U}>0, \ \forall s\in (e_0,k) \}>e_0,\]
which concludes the proof.
\end{proof}
As announced in the introduction, the assumption in Theorem~\ref{thm_positivity} that $\Vert U \Vert_{C^2}$ is small is essential. In fact, we conclude this section by presenting an example where the magnetic form is nowhere vanishing and the potential $U$ is small with respect to the $C^1$-norm but not with respect to the $C^2$-norm. To this end, we briefly introduce the interesting framework for electromagnetic systems in dimension two, where, due to dimensional reasons, the setting is considerably simplified.

\subsection{Electromagnetic curvature on surfaces} \label{ricc_surface}
Let $M$ be a closed surface which, up to passing to a double cover, we may assume to be oriented.  
If $g$ is a Riemannian metric on $M$, we denote by $\mathrm{J}^g$ and $\mathrm{vol}(g)$ the complex structure and the volume form induced by $g$, respectively.  
Recall that $\mathrm{J}^g$ and $\mathrm{vol}(g)$ are related by the identity
\[
g(\mathrm{J}^g v, w) = \mathrm{vol}(g)(v,w), \qquad \forall\, v,w\in TM.
\]
The Ricci curvature of $g$ coincides with its Gaussian curvature, which we denote by $K^g$.  
If $\sigma$ is a $2$-form on $M$, then, by dimensional reasons, it is automatically closed. Moreover, there exists a unique function $b:M\to\R$ such that
\[
\sigma = b\,\mathrm{vol}(g).
\]
We refer to the function $b$ as the \textit{magnetic function}; it provides a convenient scalar description of the magnetic form in the two-dimensional setting.  
The corresponding endomorphisms $Y^{g,b}$ and $\nabla Y^{g,b}$ then take the form
\[
Y^{g,b} = b\,\mathrm{J}^g, \qquad (\nabla Y^{g,b})_v = \mathrm{d}b(v)\,\mathrm{J}^g.
\]
Given a potential $U$ and an energy level $k>e_0$, the electromagnetic Ricci curvature reduces to the \textit{electromagnetic Gaussian curvature} 
\[
K_k^{g,b,U}:E_k\to\R,
\]
and, adapting the expression \eqref{eq:ricci_em} to this two-dimensional framework, we obtain
\begin{equation}\label{ricci_2}
\begin{split}
K_k^{g,b,U}(v) 
&= \frac{K^g}{2(k-U)}
+ \frac{\Delta_g U}{4(k-U)^2}
+ \frac{|\mathrm{d}U|_g^2}{4(k-U)^3}
- \frac{\mathrm{d}b(\mathrm{J}^g \hat v)}{(2(k-U))^{\frac32}}+ \\
&\quad
- \frac{2b\,\mathrm{d}U(\mathrm{J}^g \hat v)}{(2(k-U))^{\frac52}}
+ \frac{b^2}{4(k-U)^2}.
\end{split}
\end{equation}

\begin{ex}\label{C1ex} Consider the $2$-torus $T^2 = \mathbb{R}^2/\mathbb{Z}^2$ with coordinates $\vartheta=(\vartheta_1,\vartheta_2)$, endowed with the standard flat metric $g_0={\mathrm{d}\vartheta_1}^2 + {\mathrm{d}\vartheta_2}^2$, a constant magnetic field $b = 1$, and a scalar potential $U$ defined by
\[
U(\vartheta) = \frac{\sin(m\,\vartheta_1)}{m^j}\, , 
\]
where $m\in 2\pi \mathbb{Z}$ and $j$ a positive integer. Observe that when $j = 2$ and $m$ is sufficiently large, $\|U\|_{C^1}$ can be made arbitrarily small, while $\|U\|_{C^2}$ remains bounded away from zero; whereas for $j > 2$, the potential $U$ is $C^2$-small. 
Let $k>e_0$, let $\vartheta_0$ be a maximum point of $U$, and set $v= \sqrt{2(k-U(\vartheta_0)} \, \partial \vartheta_1$. A direct computation yields
\[
\Delta_{g_0} U(\vartheta_0) = \partial^2_{\vartheta_1} U(\vartheta_0)= -\frac{1}{m^{j-2}}, 
\]
and consequently, substituting the respective quantities into \eqref{ricci_2}, we obtain
\[
\operatorname{K}^{g_0,1,U}_k(v) = \frac{1}{4(k-U)^2} \left(1- \frac{1}{m^{j-2}} \right).
\]
Hence, for $j = 2$,
\[
\operatorname{K}^{g_0,1,U}_k(v) = 0, \quad \forall k>e_0,
\]
showing that $C^1$-smallness alone is not sufficient. Let us remark that, with minor modifications, this example can be generalized to any $n$-dimensional flat torus.

\end{ex}

\section{Existence of a closed $(g,\sigma,U)$-geodesic: proof of Theorem~\ref{thm_existence} (and Theorem~\ref{TheroemA1})}
The aim of this second section is to prove the existence results stated in Theorem~\ref{thm_existence} and Theorem~\ref{TheroemA1}. By virtue of the Jacobi–Maupertuis principle, we shall formulate the variational setting in the purely magnetic case. As a first step, we establish a correspondence between closed $(g_k,\sigma)$–geodesics and the zeros of a suitable $1$–form $\alpha_k$ defined on the space of loops with free period.

\subsection{Variational Setting}

Let $\Lambda = W^{1,2}(S^1, M)$ denote the Hilbert manifold of absolutely continuous loops $x: S^1 \to M$ with $L^2$-integrable derivatives. For each $x \in \Lambda$, the tangent space $T_x \Lambda$ at a point $x \in \Lambda$ is naturally identified with the vector space of $W^{1,2}$-sections of the pullback bundle $x^*(TM)$. We endow $\Lambda$ with the Riemannian metric
\[
g_{\Lambda,k}(\zeta, \eta) := \int_0^1 \left[ g_k(\zeta, \eta) + g_k(\nabla_x \zeta, \nabla_x \eta) \right] \mathrm{d}t, \quad \forall \zeta, \eta \in T_x \Lambda,
\]
and denote by $\vert \cdot \vert_{\Lambda,k}$ the associated norm. Let $\Lambda_0$ be the connected component of $\Lambda$ consisting of contractible loops.

Let $\mathcal{M} := \Lambda \times (0, +\infty)$,
equipped with the projection $\pi_{\Lambda} : \mathcal{M} \to \Lambda$, denote the Hilbert manifold of loops with free period. A point $(x, T) \in \mathcal{M}$ is naturally associated with a contractible loop $\gamma : [0, T] \to M$ defined by $\gamma(t) := x\left(\tfrac{t}{T}\right)$.
Depending on the context, we will use either notation $(x, T)$ or $\gamma$ to refer to an element of $\mathcal{M}$. We denote by \(\mathcal{M}_0\) the connected component of \(\mathcal{M}\) 
consisting of contractible loops.
Under the natural splitting $T\mathcal{M} = T\Lambda \oplus \mathbb{R}$,
we endow $\mathcal{M}$ with the product Riemannian metric
\[
g_{\mathcal{M},k} := g_{\Lambda,k} \oplus \mathrm{d}T^2,
\]
where \(g_{\Lambda,k}\) is the metric on $\Lambda$ defined above, 
and $\mathrm{d}T^2$ denotes the Euclidean metric on $(0, +\infty)$.

\medskip

We now introduce a smooth $1$-form $\alpha_{k,\lambda}$ on $\mathcal{M}$, whose zeros correspond to closed $(g_k, \sigma)$-geodesics. First define the 1-form $\Theta^\sigma$ on $\Lambda$ by
\[
(\Theta^\sigma)_x(\zeta) = \int_{S^1} \sigma(\dot{x}, \zeta)\,\mathrm{d}t 
= \int_{S^1} g_k(Y^{g_k, \sigma}\dot{x}, \zeta)\,\mathrm{d}t.
\]
Here $Y^{g_k, \sigma}$ denotes the Lorentz endomorphism associated with $(g_k, \sigma)$.
The form $\Theta^\sigma$ is closed in the sense that its integral over any smooth closed path $u : S^1 \to \Lambda$ depends only on the homotopy class of $u$. In particular, a local primitive $M^\sigma$ can be constructed as follows. Fix $\bar{x} \in \Lambda_0$ and let $B_{\bar{x}}(r)\subset \Lambda_0$ denote the Riemannian ball centered at $\bar{x}$ with radius $r > 0$. Define $M^\sigma : B_{\bar{x}}(r) \to \mathbb{R}$ by
\begin{equation} \label{localprimitive}
M^\sigma(x) := \int (C_{\bar{x}, x})^* \sigma,
\end{equation}
where $C_{\bar{x}, x} : [0,1] \to B_{\bar{x}}(r)$ is a smooth path connecting $\bar{x}$ and $x$. The integral on the right-hand side in equation~\eqref{localprimitive} is understood by considering $C_{\bar x,x}$ as a capping cylinder $C_{\bar x,x}:[0,1]\times S^1 \to M$ connecting $\bar x$ and $x$.
 Since $\Theta^\sigma$ is closed, the definition of $M^\sigma$ is independent of the choice of the path $C_{\bar{x}, x}$.

For $k > e_0$, we denote by $\mathcal{E}_k : \Lambda \to \mathbb{R}$ the $L^2$-energy functional associated with the metric $g_k$, defined by
\[
\mathcal{E}_k(x) = \int_{S^1} g_k(\dot{x}, \dot{x})\,\mathrm{d}t 
= \int_{S^1} 2\big(k - U(x)\big)\,|\dot{x}|_g^2\,\mathrm{d}t.
\]
Introducing an additional parameter $\lambda \in \left(-\frac{1}{2}, +\infty \right)$, let $Q_{k,\lambda} : \mathcal{M} \to \mathbb{R}$ be the smooth function defined as
\[
Q_{k,\lambda}(x,T) := \frac{\mathcal{E}_k(x)}{T} + \left(\frac{1}{2} + \lambda\right)T,
\]
 Finally, define the smooth $1$-form $\alpha_{k,\lambda}$ on $\mathcal{M}$ as
\begin{equation}\label{alphak}
\alpha_{k,\lambda} := \mathrm{d}Q_{k,\lambda} + (\pi_{\Lambda})^*\Theta^\sigma.
\end{equation}
A point $\gamma = (x, T) \in \mathcal{M}$ is called a vanishing point of $\alpha_{k,\lambda}$ if $(\alpha_{k,\lambda})_\gamma = 0$. 
By definition of $\alpha_{k,\lambda}$, such a point satisfies
\begin{equation}\label{lambdaequation}
\begin{cases}
(\mathrm{d}\mathcal{E}_k)_x = T \cdot (\Theta^\sigma)_x, \\[0.4em]
\displaystyle \frac{1}{2} + \lambda - \frac{\mathcal{E}_k(x)}{T^2} = 0.
\end{cases}
\end{equation}
It is a standard result that solutions of~\eqref{lambdaequation} correspond to smooth closed curves $\gamma$ satisfying
\begin{equation*}
\begin{cases}
\nabla^{g_k}_{\dot{\gamma}} \dot{\gamma} = Y^{g_k, \sigma}(\dot{\gamma}), \\[0.4em]
g_k(\dot{\gamma}, \dot{\gamma}) = \displaystyle \frac{1}{2} + \lambda,
\end{cases}
\end{equation*}
that is, $(g_k, \sigma)$-geodesics with prescribed energy. 

We denote by $\mathcal{Z}(\alpha_{k,\lambda})$ the set of vanishing points of $\alpha_{k,\lambda}$. By construction, $\alpha_{k,\lambda}$ is closed, and its exactness is equivalent to that of $\Theta^\sigma$. In particular, $\alpha_{k,\lambda}$ is locally exact. Using \eqref{localprimitive}, a local primitive $S_{k,\lambda}$ can be defined on $B_{\bar{x}}(r) \times (0, +\infty)$ by
\[
S_{k,\lambda}(x,T) := Q_{k,\lambda}(x,T) + M^\sigma(x).
\]
If $\gamma \in \mathcal{Z}(\alpha_{k,\lambda})$, its Morse index $\mu(\gamma)$ is defined as the Morse index of $\gamma$ with respect to any local primitive $S_{k,\lambda}$ defined on a neighborhood containing $\gamma$. This definition is independent of the choice of primitive. By \cite[Proposition~3.1]{AbboSchwarz}, the self-adjoint operator associated to the second variation of any such primitive is a compact perturbation of a positive Fredholm operator, and therefore $\mu(\gamma)$ is finite. For any non-negative integer $m$, we denote by $\mathcal{Z}^m(\alpha_{k,\lambda})$ the subset of $\mathcal{Z}(\alpha_{k,\lambda})$ consisting of vanishing points $\gamma$ with $\mu(\gamma) \leq m$.

The argument used in the proof of Theorem~\ref{thm_existence} relies on three main ingredients. The first consists in establishing a lower bound on the period for every $\gamma \in \mathcal{Z}(\alpha_{k,\lambda})$, where $\lambda$ ranges over a fixed interval.

\begin{lemma} \label{lowerboundT}
For every interval $I$ with closure contained in $(-\frac{1}{2}, +\infty)$, there exists a constant $T_I > 0$ such that if $\lambda \in I$ and $\gamma = (x, T) \in \mathcal{Z}(\alpha_{k,\lambda})$, then $T \geq T_I$.
\begin{proof}
Let $-\tfrac{1}{2} < \lambda_* < \lambda^* < +\infty$, and consider $I = (\lambda_*, \lambda^*)$. By \cite[Proposition~1.4.14]{Lecturesonclosedgeodesics}, there exist constants $\delta > 0$ and $\bar{E} > 0$ such that if $x \in \{ \mathcal{E}_k < \delta \}$, then
\begin{equation} \label{klingeberglemma}
\left| (\mathrm{d}\mathcal{E}_k)_x \right|_{\Lambda,k} \geq \bar{E} \sqrt{\mathcal{E}_k(x)}.
\end{equation}
On the other hand, by the Cauchy–Schwarz inequality, for all $x \in \Lambda$ we have:
\begin{equation} \label{magneticterm}
\left| (\Theta^\sigma)_x \right|_{\Lambda,k} \leq \Vert \sigma \Vert_{\infty,g_k} \sqrt{2 \mathcal{E}_k(x)}.
\end{equation}
Let $\lambda_n \in I$ and $\gamma_n = (x_n, T_n) \in \mathcal{Z}(\alpha_{k, \lambda_n})$ be a sequence. Suppose, by contradiction, that $T_n \to 0$. From the second equation of the vanishing point condition, we have
\[
\mathcal{E}_k(x_n) = \left( \frac{1}{2} + \lambda_n \right) T_n^2 \leq \left( \frac{1}{2} + \lambda^* \right) T_n^2,
\]
which implies that $x_n \in \{ \mathcal{E}_k < \delta \}$ for large $n$. In particular, for $n$ large $x_n$ belongs to the range where inequality \eqref{klingeberglemma} holds. Then, combining \eqref{klingeberglemma} and \eqref{magneticterm}, we obtain
\[
\bar{E} \sqrt{\mathcal{E}_k(x_n)} \leq \left| (\mathrm{d}\mathcal{E}_k)_{x_n} \right|_{\Lambda,k} = T_n \left| (\Theta^\sigma)_{x_n} \right|_{\Lambda,k} \leq T_n \Vert Y^{g_k, \sigma} \Vert_{\infty,k} \sqrt{2 \mathcal{E}_k(x_n)},
\]
which yields
\[
T_n \geq \frac{\bar{E}}{\sqrt{2} \Vert Y^{g_k, \sigma} \Vert_{\infty,k}} > 0.
\]
This contradicts the assumption that $T_n \to 0$, and the result follows.
\end{proof}
\end{lemma}
The second ingredient, which we extract from \cite[Section~4]{assenza}, is based on a Bonnet--Myers type argument that, under the assumption of positive electromagnetic curvature, allows one to estimate the period of $\gamma \in \mathcal{Z}(\alpha_{k,\lambda})$ in terms of its Morse index $\mu(\gamma)$, previously defined. To this end, we briefly introduce a preliminary framework, which can be found in detail in the previously mentioned reference.

\subsection{A Bonnet-Myers type argument (see \cite[Section~4]{assenza})}
First, we adapt to the $\lambda$-parametrization the definition of magnetic Ricci curvature given in~\eqref{riccimagnetica}. We define $\mathrm{Ric}_\lambda^{g_k,\sigma}:S^{g_k}M\to \R$ as
\begin{equation}\label{lambdaricc}
\mathrm{Ric}_\lambda^{g_k,\sigma}(v)
= (1+2\lambda)\,\mathrm{Ric}^{g_k}(v)
- \sqrt{1+2\lambda}\,\operatorname{trace}\big((\nabla Y^{g_k,\sigma})_v\big)
+ \operatorname{trace}\big((\widetilde{Y}^{g_k,\sigma})_v\big).
\end{equation}
Clearly, one has $\mathrm{Ric}_0^{g_k,\sigma}=\mathrm{Ric}^{g_k,\sigma}$. In fact, $\mathrm{Ric}_\lambda^{g_k,\sigma}$ is precisely the magnetic curvature of the system $(g_k,\sigma)$ on the energy level $\lambda+\frac{1}{2}$. \\

Let $\gamma \in \mathcal{Z}(\alpha_{k,\lambda})$ and write $\speed_k=\frac{\speed}{\vert \speed \vert_{g_k}}.$ Consider the splitting of $\gamma^*TM$ given by
\[
\gamma^*TM = \R \speed \oplus \{\speed\}^\perp.
\]
If $V$ is a vector field along $\gamma$, then $V = V_\parallel + V_\perp$, where $V_\parallel$ and $V_\perp$ denote, respectively, the tangential and the perpendicular components of $V$ with respect to $\speed$. Let $S_{k,\lambda}$ be a primitive of $\alpha_{k,\lambda}$ defined on a neighborhood of $\gamma$. By \cite[Lemma~9]{assenza}, the second variation $(\mathrm{d}^2S_{k, \lambda})_\gamma$ of $S_{k,\lambda}$ at $\gamma$ evaluated at $\zeta=(V,\tau)\in T_\gamma \M$, with $V$ smooth, read as
\begin{equation}\label{hess}
\begin{split}
(\mathrm{d}^2S_{k, \lambda})_\gamma\Big[\zeta,\zeta\Big] = & \int_0^T \lvert (\dot V)_\perp - (A^{g_k,\sigma}V)_\perp \rvert_{g_k}^2 \, \mathrm{d}t 
+  \int_0^T \left[ g_k(\dot V, \dot \gamma_k) - \frac{\tau}{T}\sqrt{\lambda+\frac12} \, \right]^2 \mathrm{d}t  \, + \\
& \qquad\qquad\qquad - \int_0^T \lvert V_\perp \rvert_{g_k}^2 \, \mathrm{Sec}_{\lambda}^{g_k, \sigma} \left( \speed_k, \frac{V_\perp}{\lvert V_\perp \rvert_{g_k}} \right) \mathrm{d}t,
\end{split}
\end{equation}
where, in the first term, $A^{g_k,\sigma}$ denotes the skew-adjoint endomorphism of $\gamma^*TM$ defined by
\begin{equation}\label{paralleltransp}
A^{g_k,\sigma}V=Y^{g_k, \sigma}V_\parallel + (Y^{g_k, \sigma}V)_\parallel + \tfrac{1}{2} (Y^{g_k, \sigma}V_\perp)_\perp,
\end{equation}
while, denoting by $\mathrm{Sec}^{g_k}$ the sectional curvature of the metric $g_k$, in the third term we write
\begin{equation}\label{sectionalcurvature}
\begin{split}
\mathrm{Sec}_{\lambda}^{g_k, \sigma}\left( \speed_k, \frac{V_\perp}{\lvert V_\perp \rvert_{g_k}} \right) 
&= (1+\lambda)\, \mathrm{Sec}^{g_k}\left( \speed_k,  \frac{V_\perp}{\lvert V_\perp \rvert_{g_k}} \right) + \\
&\qquad\qquad- \sqrt{1+\lambda}\, g_k\!\left((\nabla Y^{g_k,\sigma})_{\speed_k}\frac{V_\perp}{\lvert V_\perp \rvert_{g_k}},\frac{V_\perp}{\lvert V_\perp \rvert_{g_k}} \right) +\\
&\qquad +\, g_k \!\left( {(\widetilde{Y}^{g_k,\sigma})_{\speed_k}}\frac{V_\perp}{\lvert V_\perp \rvert_{g_k}},\frac{V_\perp}{\lvert V_\perp \rvert_{g_k}}  \right).
\end{split}
\end{equation}
In the following remark, we recall a method for constructing variations along $\gamma$ that significantly simplify the expression \eqref{hess} of $(\mathrm{d}^2S_{k, \lambda})_\gamma$. 
\begin{rmk}\label{remark}
Let $W$ be a unit vector field along $\gamma$, orthogonal to $\speed$, satisfying the differential equation 
\begin{equation}\label{magnetictransport}
\dot W = A^{g_k,\sigma} W,
\end{equation}
where $A^{g_k,\sigma}$ has been defined in~\eqref{paralleltransp}. If $f:[0,T] \to \R$ is a function such that $f(0)=f(T)=0$, we set $W^f=fW$. Now let $h:[0,T] \to \R$ and $\tau \in \R$ be such that the pair $(h,\tau)$ is the unique solution of the differential problem 
\[
\begin{cases}
\dot h + g_k(\dot W, \dot \gamma_k) - \dfrac{\tau}{T}\sqrt{\lambda+\frac12} = 0, \\[0.7em]
h(0)=h(T)=0.
\end{cases}
\]
By construction, if we define $\zeta=(W^f+h\speed_k, \tau)$, a direct computation shows that 
\begin{equation}\label{Hess2}
   (\mathrm{d}^2S_{k, \lambda})_\gamma\Big[\zeta,\zeta\Big] =  \int_0^T \left\{ \dot f^2 - f^2 \mathrm{Sec}_\lambda^{g_k, \sigma} \left( \speed_k, W \right) \right\} \, dt.
\end{equation}
\end{rmk}
Thanks to the framework established in the previous remark, we are now able to proceed to the central result of our argument.

\begin{lemma} \label{bonnetmayerslemma}
If $\mathrm{Ric}_\lambda^{g_k, \sigma} \geq \frac{1}{r^2} > 0$ for some $r > 0$, then, for any $\gamma =(x,T) \in \mathcal{Z}(\alpha_{k,\lambda})$, we have
\[
T \leq \pi r (\mu(\gamma) + 1).
\]
\begin{proof}
Let $\gamma \in \mathcal{Z}(\alpha_{k,\lambda})$ be such that $\mu(\gamma)=m$ for some nonnegative integer, and assume by contradiction that $T>\pi r(m+1)$. Let $W_1,\dots,W_{n-1}$ be unit, pairwise orthogonal vector fields along $\gamma$, each orthogonal to $\speed$, and satisfying Equation~\eqref{magnetictransport}. 

First, observe that by the definition~\eqref{lambdaricc} of $\mathrm{Ric}_\lambda^{g_k,\sigma}$ and by~\eqref{sectionalcurvature}, we have
\begin{equation}\label{C1}
\sum_{i=1}^{n-1} \mathrm{Sec}_\lambda^{g_k,\sigma}\left(\speed_k,W_i\right) = \mathrm{Ric}_\lambda^{g_k,\sigma}(\speed_k).
\end{equation}

Now, for $i=1,\dots,n-1$ and $j=0,\dots,m$, consider the variation
\[
\zeta_{ij} =({W_i}^{f_j}+h_{ij}\speed_k,\tau_{ij}),
\]
where
\[
f_j(t) = 
\begin{cases}
\sin \left( \dfrac{(m + 1)\pi t}{T}\right), & t \in \left[\dfrac{jT}{m+1}, \dfrac{(j+1)T}{m+1}\right], \\[0.5em]
0, & \text{otherwise},
\end{cases}
\]
and $h_{ij}$ and $\tau_{ij}$ are chosen as in Remark~\ref{remark}. 

By~\eqref{C1} and~\eqref{Hess2}, we obtain
\begin{equation*}
\begin{split}
\sum_{i=1}^{n-1} (\mathrm{d}^2S_{k, \lambda})_\gamma\Big[\zeta_{ij},\zeta_{ij}\Big] 
& = \sum_{i=1}^{n-1} \int_{\frac{jT}{m+1}}^{\frac{(j+1)T}{m+1}}  \frac{(m+1)^2 \pi^2}{T^2} 
\cos^2 \left( \frac{(m+1)\pi t}{T} \right) \, \mathrm{d}t \, +\\
& \quad - \sum_{i=1}^{n-1} \int_{\frac{jT}{m+1}}^{\frac{(j+1)T}{m+1}} 
\sin^2 \left( \frac{(m+1)\pi t}{T} \right) 
\mathrm{Sec}_\lambda^{g_k, \sigma}( \speed_k, W_i)\, \mathrm{d}t \\
& = (n-1)\left[ \frac{(m+1)^2 \pi^2}{2T(m+1)} 
- \int_{\frac{jT}{m+1}}^{\frac{(j+1)T}{m+1}}  
\sin^2 \left( \frac{(m+1)\pi t}{T} \right) 
\mathrm{Ric}_\lambda^{g_k, \sigma}(\speed_k)\, \mathrm{d}t \right] \\
& \le (n-1) \left[ \frac{(m+1)^2 \pi^2}{2T(m+1)} - \frac{T}{2 r^2(m+1)}\right] \\
& = (n-1) \left[ \frac{(m+1)^2 \pi^2r^2-T^2}{2T(m+1)r^2}\right] < 0.
\end{split}
\end{equation*}

Therefore, for each $j$, there exists $i_j \in \{1, \dots, n-1 \}$ such that 
\[
(\mathrm{d}^2S_{k, \lambda})_\gamma\Big[\zeta_{i_jj},\zeta_{i_jj}\Big] < 0.
\]
One can show that the span of the vectors $\zeta_{i_jj}$ forms an $(m+1)$–dimensional subspace of $T_{\gamma}\M$ on which $(\mathrm{d}^2S_{k, \lambda})_\gamma$ is negative definite (see \cite[Lemma~14]{assenza}). This contradicts the assumption on $\mu(\gamma)$.
\end{proof}
\end{lemma}
Let us point out that the inequality in the statement of Lemma~\ref{bonnetmayerslemma} is sharp, as illustrated by the example of the round $2$–sphere or the flat $2$–torus endowed with a non–identically zero constant magnetic function.

\subsection{Proof of Theorem \ref{thm_existence} (and Theorem~\ref{TheroemA1})}
The final ingredient we need is the existence of a contractible zero of $\alpha_{k,\lambda}$ for almost every $\lambda$ approaching $0$ from below.
\begin{lemma} \label{struwelemma}
There exists $\lambda_0 > 0$ and a subset $J \subseteq \left(- \lambda_0,  0 \right)$ of full Lebesgue measure such that for every $\lambda \in J$, the set $\mathcal{Z}^1(\alpha_{k,\lambda}) \cap \M_0$ is non-empty.
\end{lemma}
As mentioned several times, this result follows from Struwe’s monotonicity argument~\cite{Struwe}, adapted to the minimax geometry of $\alpha_{k,\lambda}$ arising near the set of constant loops. Since this construction is standard and has been employed previously by several authors, we have decided to include it in Appendix~\ref{Appendix}.

We are now ready to prove Theorem~\ref{thm_existence}.

\begin{proof}[Proof of Theorem \ref{thm_existence} (and Theorem \ref{TheroemA1})]
Let $k\in (e_0,c]$ be such that $\mathrm{Ric}^{g,\sigma,U}_k > 0$, i.e. $\mathrm{Ric}_0^{g_k, \sigma} > 0$. By continuity, we can find an open interval $I_0$ containing $0$ and a constant $C > 0$ such that
\[
\mathrm{Ric}_\lambda^{g_k, \sigma} \geq \frac{1}{C^2}, \quad \forall \, \lambda \in I_0.
\]
By Lemma \ref{struwelemma}, we obtain sequences $\lambda_n \in J \cap I_0$ and $\gamma_n = (x_n, T_n) \in \mathcal{Z}^1(\alpha_{k,\lambda_n}) \cap \M_0$ such that $\lambda_n \nearrow 0$ and each $\gamma_n$ satisfies:
\[
\begin{cases}
(\mathrm{d}\mathcal{E}_k)_{x_n} = T_n \cdot (\Theta^\sigma)_{x_n}, \\
\frac{1}{2} + \lambda_n - \frac{\mathcal{E}_k(x_n)}{T_n^2} = 0.
\end{cases}
\]
By Lemmas~\ref{lowerboundT} and~\ref{bonnetmayerslemma}, the sequence $T_n$ is uniformly bounded and bounded away from zero. Thus, by the Ascoli--Arzel\`a theorem, up to a subsequence, $\gamma_n$ converges uniformly to $\bar{\gamma} = (\bar{x}, \bar{T})$. Since the sequence $\gamma_n$ consists of solutions of a one-parameter ordinary differential equation, it follows that $\gamma_n$ converges to $\bar{\gamma}$ with respect to the $C^1$-norm. Hence, $\gamma_n$ converges to $\bar{\gamma}$ in $\M$, and $\bar{\gamma}$ satisfies
\[
\begin{cases}
(\mathrm{d}\mathcal{E}_k)_{\bar{x}} = \bar{T} \cdot (\Theta^\sigma)_{\bar{x}}, \\
\frac{1}{2} - \frac{\mathcal{E}_k(\bar{x})}{\bar{T}^2} = 0,
\end{cases}
\]
i.e., $\bar{\gamma} \in \mathcal{Z}(\alpha_{k,0}) \cap \M_0$. Composing $\bar{\gamma}$ with the reparametrization $s(t) = \int_0^t 2(k-U(\gamma(\tau))) \mathrm{d}\tau$, we obtain a contractible solution of \eqref{Newtonequation} with energy $k$, concluding the proof of Theorem~\ref{thm_existence}. Finally, if $\sigma$ is nowhere vanishing and $U$ small with respect to the $C^2$ norm, then $c>e_0$ and by Theorem \ref{thm_positivity}, the value $\nu_{\mathrm{Ric}}>e_0$. By setting,
\[e_0<\nu_0 < \min\{ c, \nu_{\mathrm{Ric}}\},\]
Theorem \ref{TheroemA1} follows.
\end{proof}

\subsection*{Acknowledgment}
V.A was supported by the Instituto Serrapilheira. G.T. is partially funded by the CRC/TRR 191 ``Symplectic structures in Geometry, Algebra and Dynamics'' and by the DFG project 566804407 ``Symplectic Dynamics, Celestial Mechanics and Magnetism''. The groundwork for this article was carried out during the CIME School on \emph{Symplectic Dynamics and Topology}, held in Cetraro in June 2025. The authors thank the organizers for their hospitality and the stimulating atmosphere. The authors also wish to thank L.~Asselle, G.~Benedetti and L.~Macarini for their valuable suggestions.
 
\appendix
\section{Struwe monotonicity argument}\label{Appendix}
In this appendix we outline the Struwe’s monotonicity argument
used in the proof of Theorem~\ref{thm_existence}. We discuss in detail the cases in which the magnetic form $\sigma$ is weakly exact, and briefly explain how to adapt the framework to the non–weakly exact case.

\subsection{Weakly exact case}
In this subsection we assume that the 2-form $\sigma$ is weakly exact. 
This hypothesis ensures the existence of a global primitive $M^\sigma$ of $\Theta^\sigma$ on $\Lambda_0$, defined by
\[
M^\sigma(x) := \int (D_x)^* \sigma,
\]
where $D_x$ is a smooth capping disk for the loop $x$. 
Since $\sigma$ integrates to zero over all spheres, the definition above is independent of the choice of $D_x$. 
Therefore, one obtains a global primitive $A_{k,\lambda} : \mathcal{M}_0 \to \mathbb{R}$ of $\alpha_{k,\lambda}$ given by
\begin{equation} \label{globalprimitive}
A_{k,\lambda}(x,T) := Q_{k,\lambda}(x,T) + M^\sigma(x).
\end{equation}

Within this variational framework, the Mañé critical value $c$ can then be characterized as

\begin{equation} \label{manevar}
c:= \inf \left\{ k \geq e_0 \,\middle|\, A_{k,0}(x,T) > 0 \right\}.
\end{equation}
With a slight abuse of notation, we identify $M$ with the subset of $\Lambda_0$ consisting of constant loops, and set $M^+ := M \times (0, +\infty)$ as its counterpart in $\mathcal{M}_0$. If $k \in (e_0, c)$, then, by the definition \eqref{manevar} of $c$ and by the continuity of $A_{k,\lambda}$, there exists $\lambda_0 \in \left(0, \tfrac{1}{2} \right)$ such that for every $\lambda \in (-\lambda_0, \lambda_0)$, the following set is non-empty:
\[
\Gamma_{\lambda} := \left\{ \varphi : [0,1] \to \mathcal{M}_0 \,\middle|\, \varphi(0) \in M^+, \ \varphi(1) \in \left\{ A_{k,\lambda} < 0 \right\} \right\} \neq \emptyset.
\]
We define the minimax value function $u : (-\lambda_0, \lambda_0) \to \mathbb{R}$ by
\[
u(\lambda) := \inf_{\varphi \in \Gamma_\lambda} \max_{s \in [0,1]} A_{k,\lambda}(\varphi(s)).
\]

The mountain pass structure underlying this construction is summarized in the following lemma.

\begin{lemma} \label{minimaxvalue}
The function $u$ is monotone non-decreasing. Moreover, there exists a constant $D > 0$ such that $u > D$.
\begin{proof}
The monotonicity of $u$ follows directly from the fact that $A_{k,\lambda}$ is monotone non-decreasing with respect to the parameter $\lambda$. Observe that for every $\lambda \in (-\lambda_0, \lambda_0)$ and for every $(x,T) \in \mathcal{M}_0$, the following inequality holds:
\begin{equation} \label{ineQ}
Q_{k,\lambda}(x,T) \geq 2\sqrt{\tfrac{1}{2} + \lambda} \sqrt{\mathcal{E}_k(x)} \geq 2\sqrt{\tfrac{1}{2} - \lambda_0} \sqrt{\mathcal{E}_k(x)}.
\end{equation}
Moreover, if $x$ is entirely contained in an open subset of $M$ diffeomorphic to a disk, then by \cite[Lemma~7.1]{abbo2013} we also have:
\begin{equation} \label{ineM}
\left| M^\sigma(x) \right| \leq \frac{\| Y^{g_k, \sigma} \|_{\infty,k}}{2} \mathcal{E}_k(x),
\end{equation}
which implies:
\[
A_{k,\lambda}(x,T) \geq 2\sqrt{\tfrac{1}{2} - \lambda_0} \sqrt{\mathcal{E}_k(x)} - \frac{\| Y^{g_k} \|_{\infty,k}}{2} \mathcal{E}_k(x).
\]
From the above inequality and the definition of $\Gamma_\lambda$, we can deduce that for every sufficiently small $\delta > 0$, it holds:
\[
\varphi([0,1]) \cap \mathcal{E}_k^{-1}(\delta) \neq \emptyset, \quad \forall \varphi \in \Gamma_\lambda.
\]
The statement follows.
\end{proof}
\end{lemma}

A pseudo-gradient for \( A_{k,\lambda} \) that leaves the set \( \Gamma_\lambda \) invariant can be constructed as follows. Choose a function \( h_{\lambda}: \mathbb{R} \to [0,1] \) with \( h_{\lambda}' \geq 0 \) such that
\[
\begin{cases}
  h_\lambda(t) = 0, & t \in (-\infty, \tfrac{u(\lambda)}{2}] \\
h_\lambda(t) = 1 & t \in [\tfrac{u(\lambda)}{2}, +\infty)  
\end{cases}
\]
%\[
%h_{\lambda}^{-1}\left((-\infty, \tfrac{u(\lambda)}{2}]\right) = 0 \quad \text{and} \quad h_{\lambda}^{-1}\left([\tfrac{u(\lambda)}{2}, +\infty)\right) = 1.
%\]
Write for simplicity $\nabla A_{k,\lambda} = \mathrm{grad}^{g_{\mathcal{M},k}}A_{k,\lambda}$, and define the vector field \( X_{k,\lambda} \) by
\begin{equation} \label{Pseudograd}
X_{k,\lambda} = -\frac{(h_\lambda \circ A_{k,\lambda})}{\sqrt{1 + \vert \nabla A_{k,\lambda} \vert_{\mathcal{M},k}^2}} \nabla A_{k,\lambda}.
\end{equation}
By \cite[Lemma~5.7]{Merry} and \cite[Remark~1.4]{abbo2013}, the positive semi-flow
\[
F_{k,\lambda}: \mathcal{M}_0 \times [0, +\infty) \to \mathcal{M}_0
\]
obtained by integrating \( X_{k,\lambda} \) is complete. We summarize the properties of \( F_{k,\lambda} \) that will be needed in our argument in the following lemma.

\begin{lemma} \label{gradientproperties}
Let \( \eta: [0, +\infty) \to \mathcal{M}_0 \) be a flow line of \( F_{k,\lambda} \), and write \( \eta(s) = (x_s, T_s) \). The following hold:
\begin{itemize}
    \item[(i)] For every \( s \geq 0 \), it holds that
    \[
    \frac{\mathrm{d}}{\mathrm{d}s} A_{k,\lambda}(\eta(s)) \leq 0.
    \]
    
    \item[(ii)] If for some \( s_* \geq 0 \), we have \( \eta(s_*) \in \{ A_{k,\lambda} \geq \tfrac{u(\lambda)}{2} \} \), then
    \begin{equation} \label{boundX_{klambda}}
    \frac{\mathrm{d}}{\mathrm{d}s} A_{k,\lambda}(\eta(s)) = -\frac{\vert \nabla A_{k,\lambda} \vert_{\mathcal{M},k}^2}{\sqrt{1 + \vert \nabla A_{k,\lambda} \vert_{\mathcal{M},k}^2}}, \quad \forall s \in [0, s_*].
    \end{equation}

    \item[(iii)] For every \( s \geq 0 \), it holds that
    \[
    T_s \leq T_0 + \sqrt{s \left( A_{k,\lambda}(\eta(0)) - A_{k,\lambda}(\eta(s)) \right)}.
    \]
\end{itemize}

\begin{proof}
Points (i) and (ii) follow from the fact that \( h_\lambda \) is non-negative and identically equal to 1 on the region \( \{ A_{k,\lambda} \geq \tfrac{u(\lambda)}{2} \} \). By the definition of \( X_{k,\lambda} \) and the Cauchy–Schwarz inequality, we also obtain:
\begin{align*}
    s \left( A_{k,\lambda}(\eta(0)) - A_{k,\lambda}(\eta(1)) \right)
    &= -s \int_0^s (\mathrm{d}A_{k,\lambda})_{\eta(\tau)} \left( \frac{\mathrm{d}}{\mathrm{d}\tau} \eta(\tau) \right) \, \mathrm{d}\tau \\
    &\geq s \int_0^s \left\vert \frac{\mathrm{d}}{\mathrm{d}\tau} \eta(\tau) \right\vert_{{\mathcal{M},k}}^2 \, \mathrm{d}\tau \\
    &\geq \left( \int_0^s \left\vert \frac{\mathrm{d}}{\mathrm{d}\tau} \eta(\tau) \right\vert_{{\mathcal{M},k}} \, \mathrm{d}\tau \right)^2 \\
    &\geq d_{\mathcal{M},k}(u(s), u(0))^2.
\end{align*}
Here, \( d_{\mathcal{M},k} \) denotes the distance on \( \mathcal{M} \) induced by the metric \( g_{\mathcal{M}} \). Point (iii) then follows by observing:
\[
T_s \leq \vert T_s - T_0 \vert + T_0 \leq T_0 + d_{\mathcal{M},k}(\eta(s), \eta(0)).
\]
\end{proof}
\end{lemma}

Denote by $\mathrm{Crit}^1(A_{k,\lambda})$ the set of critical points of $A_{k,\lambda}$ whose Morse index is at most one. 
Lemma~\ref{struwelemma} follows from the following result. 

\begin{lemma}
There exists a subset $J \subseteq (-\lambda_0, \lambda_0)$ of full Lebesgue measure such that, for every $\lambda \in J$, the set $\mathrm{Crit}^1(A_{k,\lambda}) \cap \{ A_{k,\lambda} = u(\lambda)\} \neq \emptyset$. In particular, $\mathcal{Z}^1(\alpha_{k,\lambda})\cap\M_0\neq \emptyset$.
\begin{proof}
By Lemma~\ref{minimaxvalue}, the function \( u \) is monotone non-decreasing, which implies that it is differentiable on a set \( J \subseteq (-\lambda_0, \lambda_0) \) of full Lebesgue measure. Thus, for every \( \lambda \in J \), there exists a constant \( D_\lambda \) such that for every \( \lambda' \) sufficiently close to \( \lambda \), we have:
\begin{equation} \label{constantDlambda}
\vert u(\lambda) - u(\lambda') \vert \leq D_\lambda \vert \lambda - \lambda' \vert.
\end{equation}
Fix \( \lambda \in J \) and consider a sequence \( \lambda_n\subset (-\lambda_0,\lambda_0)\) such that \( \lambda_n \searrow \lambda \), and define \( \varepsilon_n := \lambda_n - \lambda \searrow 0 \). By the definition of \( u \), for each \( n \) there exists \( \varphi_n \in \Gamma_\lambda \) such that:
\[
\max_{s \in [0,1]} A_{k, \lambda_n}(\varphi_n(s)) \leq u(\lambda_n) + \varepsilon_n.
\]
Observe that for each \( n \), the following inclusion holds:
\[
\Gamma_{\lambda_n} \subseteq \Gamma_{\lambda_{n+1}} \subseteq \Gamma_\lambda.
\]
Now, if \( \gamma = (x, T) \in \varphi_n([0,1]) \) satisfies \( A_{k,\lambda}(\gamma) \geq u(\lambda) - \varepsilon_n \), then by \eqref{constantDlambda}:
\begin{equation*}
    T = \frac{A_{k,\lambda_n}(\gamma) - A_{k,\lambda}(\gamma)}{\varepsilon_n}
    \leq \frac{u(\lambda_n) - u(\lambda) + 2\varepsilon_n}{\varepsilon_n}
    \leq D_\lambda + 2,
\end{equation*}
and
\begin{equation*}
    A_{k,\lambda}(\gamma) \leq A_{k, \lambda_n}(\gamma) \leq u(\lambda_n) + \varepsilon_n \leq u(\lambda) + \varepsilon_n(D_\lambda + 1).
\end{equation*}
Using these estimates and points (i) and (iii) of Lemma~\ref{gradientproperties}, we conclude that for every \( \gamma \in \varphi_n([0,1]) \),
\[
F_{k,\lambda}(\gamma, [0,1]) \subset \{ A_{k,\lambda} < u(\lambda) - \varepsilon_n \} \cup \mathcal{C}_n,
\]
where we define:
\[
\mathcal{C}_n := \left\{ u(\lambda) - \varepsilon_n \leq A_{k,\lambda} \leq u(\lambda) + \varepsilon_n(D_\lambda + 2) \right\}
\cap \left\{ T < D_\lambda + 2 + \sqrt{\varepsilon_n(D_\lambda + 2)} \right\}.
\]

We claim that there exists a sequence \( \gamma_n \in \mathcal{C}_n \) such that:
\[
\vert (\mathrm{d}A_{k,\lambda})_{\gamma_n} \vert_{\mathcal{M},k} \to 0, \quad \text{and} \quad A_{k,\lambda}(\gamma_n) \to u(\lambda).
\]
We argue by contradiction. Suppose instead that \( \vert \nabla A_{k,\lambda} \vert_{\mathcal{M},k} \) is bounded away from zero on \( \mathcal{C}_n \) for large \( n \). Then, there exists \( \delta > 0 \) such that:
\begin{equation} \label{boundgrad}
\frac{\vert \nabla A_{k,\lambda} \vert_{\mathcal{M},k}^2}{\sqrt{1 + \vert \nabla A_{k,\lambda} \vert_{\mathcal{M},k}^2}} \geq \delta, \quad \forall \gamma \in \mathcal{C}_n.
\end{equation}
If \( \gamma \in \varphi_n([0,1]) \) and \( F_{k,\lambda}(\gamma, 1) \in \{ A_{k,\lambda} \geq u(\lambda) \} \), then by Lemma~\ref{gradientproperties} (ii) and \eqref{boundgrad}, we obtain
\begin{align*}
A_{k,\lambda}(F_{k,\lambda}(\gamma,1)) 
&= A_{k,\lambda}(\gamma) - \int_0^1 \frac{d}{d\tau} A_{k,\lambda}(F_{k,\lambda}(\gamma, \tau)) \, d\tau \\
&\leq u(\lambda) + \varepsilon_n(D_\lambda + 1) - \delta.
\end{align*}
For \( n \) sufficiently large, this contradicts the definition of \( u \), proving the claim.

It remains to show that the sequence \( \gamma_n = (x_n, T_n) \) has a convergent subsequence in \( \mathcal{M}_0 \). By \cite[Lemma~5.3]{abbo2013}, this is true provided that \( T_n \) is uniformly bounded and bounded away from zero. The former is immediate from \( \gamma_n \in \mathcal{C}_n \). For the latter, since \( \vert (\mathrm{d}A_{k,\lambda})_{\gamma_n} \vert_{\mathcal{M},k} \to 0 \), we have:
\begin{equation} \label{infinitesimal}
\frac{1}{2} + \lambda_n - \frac{\mathcal{E}_k(x_n)}{T_n^2} = \beta_n,
\end{equation}
with \( \beta_n \to 0 \) and \( \lambda_n \to \lambda \). Then,
\[
\mathcal{E}_k(x_n) = \left( \frac{1}{2} + \lambda_n - \beta_n \right) T_n^2.
\]
If \( T_n \to 0 \), then \( \mathcal{E}_k(x_n) = o(T_n^2) \), and from \eqref{ineM}, we deduce:
\[
\left\vert A_{k,\lambda}(x_n, T_n) \right\vert = o(T_n).
\]
Since \( A_{k,\lambda}(x_n, T_n) \to u(\lambda) > 0 \), we conclude that \( T_n \) is bounded away from zero.
Therefore, \( \gamma_n \to \bar{\gamma} \in \mathrm{Crit}(A_{k,\lambda}) \cap \{ A_{k,\lambda} = u(\lambda) \} \), and standard minimax theory (see \cite[Lemma~20]{assenza}) implies that \( \bar{\gamma} \) has Morse index at most one. Hence,
\[
\mathcal{Z}^1(\alpha_{k,\lambda}) \cap \mathcal{M}_0 \neq \emptyset.
\]
\end{proof}
\end{lemma}
\subsection{Non weakly exact case}
If $\sigma$ is not exact on the universal cover, then the critical value $c = +\infty$, and neither $\Theta^\sigma$ nor $\alpha_{k,\lambda}$ admits a globally defined primitive. In particular, the function $A_{k,\lambda}$ defined in \eqref{globalprimitive} is only locally defined in a neighborhood of $M^+$. Nevertheless, for any smooth path $\varphi:[0,1] \to \R$, we can still define the variation of $\alpha_{k,\lambda}$ along $\varphi$, denoted $\Delta \alpha_{k,\lambda} : [0,1] \to \R$, as
\[
\Delta \alpha_{k,\lambda}(\varphi)(s) := \int_0^s \varphi^* \alpha_{k,\lambda}.
\]
It is clear that if the image of $\varphi$ lies entirely within a domain where a local primitive $S_{k,\lambda}$ of $\alpha_{k,\lambda}$ is defined, then
\[
\Delta \alpha_{k,\lambda}(\varphi)(s) = S_{k,\lambda}(\varphi(s)) - S_{k,\lambda}(\varphi(0)).
\]

In this framework, the minimax geometry can be formulated as follows. First, recall the one-to-one correspondence
\[
\left\{ \varphi : S^2 \to M \right\} \longleftrightarrow \left\{ \varphi : [0,1] \to \Lambda_0 \ ; \ \varphi(0), \varphi(1) \in M \right\},
\]
which descends naturally to homotopy classes. For a nontrivial class $[a] \in \pi_2(M) \setminus \{0\}$, define
\[
\Gamma_{[a]} := \left\{ \varphi:[0,1] \to \M_0 \ ; \ \varphi(0),\varphi(1) \in M^+, \ \text{and} \ (\pi_{\Lambda} \circ \varphi) \in [a] \right\}.
\]

Fix a small constant $\delta > 0$, and let $\mathcal{V}_\delta = E^{-1}([0,\delta)) \times (0,+\infty)$. Define $A_{k,\lambda} : \mathcal{V}_\delta \to \R$ as in \eqref{globalprimitive}. For each $\varphi \in \Gamma_{[a]}$, define a primitive $S_{k,\lambda}(\varphi)$ of $\alpha_{k,\lambda}$ along $\varphi$ by
\[
S_{k,\lambda}(\varphi)(s) := \Delta \alpha_{k,\lambda}(\varphi)(s) + A_{k,\lambda}(\varphi(0)).
\]

Let $\lambda_0 < \frac{1}{2}$ and define the function $u : (-\lambda_0, \lambda_0) \to \R$ by
\[
u(\lambda) := \inf_{\varphi \in \Gamma_{[a]}} \max_{s \in [0,1]} S_{k,\lambda}(\varphi)(s).
\]
As in the weakly exact case, the function $u$ is monotone non-decreasing. Furthermore, since $\Delta \alpha_{k,\lambda}$ coincides with $A_{k,\lambda}$ for paths fully contained in $\mathcal{V}_\delta$, and because $[a] \neq 0$, a straightforward adaptation of the argument used in Lemma \ref{minimaxvalue} shows that $u > D$ for some constant $D > 0$.

As a pseudo-gradient, we may consider the vector field $X_{k,\lambda}$ obtained from \eqref{Pseudograd} by replacing $\nabla A_{k,\lambda}$ with $\sharp \alpha_{k,\lambda}$, the vector field on $\mathcal{M}_0$ dual to $\alpha_{k,\lambda}$ under the natural pairing between $T^* \M$ and $T \M$ induced by $g_{\M,k}$. Additionally, the term $(h_\lambda \circ A_{k,\lambda})$ is replaced with $(q_\lambda \circ \mathcal{E}_k)$, where $q$ is a smooth, monotone increasing function satisfying
\[
q_\lambda^{-1}\left((-\infty, \tfrac{\delta}{4}]\right) =  0, \quad \text{and} \quad q_\lambda^{-1}\left([\tfrac{\delta}{2}, +\infty)\right) = 1.
\]

The positive semi-flow $F_{k,\lambda}$ generated by integrating $X_{k,\lambda}$ is complete and preserves $\Gamma_{[a]}$. For every $\gamma \in \M_0$, if $\eta_\gamma:[0,+\infty) \to \M_0$ denotes the flow line of $F_{k,\lambda}$ starting at $\gamma$, then
\[
\Delta \alpha_{k,\lambda}(\eta_\gamma)(s) \leq 0.
\]
This inequality, together with the closedness of $\alpha_{k,\lambda}$, implies that for every $\varphi \in \Gamma_{[a]}$ and all $t \geq 0$,
\[
S_{k,\lambda}(F_{k,\lambda}(t, \varphi))(s) = S_{k,\lambda}(\varphi)(s) + \Delta \alpha_{k,\lambda}(\eta_{\varphi(s)})(t) \leq S_{k,\lambda}(\varphi)(s).
\]

Moreover, item (iii) of Lemma \ref{gradientproperties} also extends to this setting. With these tools in place, the proof of Lemma \ref{struwelemma} for the case where $\sigma$ is not weakly exact proceeds along the same lines as in the weakly exact case. Applying Struwe's monotonicity argument to the function $u$, we find that for every point $\lambda$ at which $u$ is differentiable, there exists a sequence $\gamma_n$ such that:
\[
\left| \alpha_{k,\lambda}(\gamma_n) \right|_{\M,k} \to 0, \quad \text{and} \quad \gamma_n \in \left\{ T \leq D_k + 2 + \sqrt{\varepsilon_n(D_k + 2)} \right\},
\]
for $\varepsilon_n = \lambda_n-\lambda \searrow 0$. In this case, the sequence $T_n$ is uniformly bounded above by construction and uniformly bounded away from zero, since for every $\varphi \in \Gamma_{[a]}$, the maximum of $S_{k,\lambda}(\varphi)$ occurs outside of $\mathcal{V}_{\delta_1}$ for some $\delta_1 \in (0, \delta)$.

Therefore, by \cite[Theorem~2.6]{lus-fet}, the sequence $\gamma_n$ converges (up to subsequence) to a point $\bar{\gamma} \in \mathcal{Z}(\alpha_{k,\lambda}) \cap \M_0$. Finally, the implication that $\mathcal{Z}^1(\alpha_{k,\lambda}) \cap \M_0 \neq \emptyset$ follows by small adaptation of the same argument used in the weakly exact case (see \cite[Lemma~25]{assenza} for further details).

\bibliographystyle{plain}
\bibliography{biblio}

\begin{thebibliography}{10}

\bibitem{abbo2013}
Alberto Abbondandolo.
\newblock Lectures on the free period {Lagrangian} action functional.
\newblock {\em J. Fixed Point Theory Appl.}, 13(2):397--430, 2013.

\bibitem{AMMP1}
Alberto Abbondandolo, Leonardo Macarini, Marco Mazzucchelli, and Gabriel~P. Paternain.
\newblock Infinitely many periodic orbits of exact magnetic flows on surfaces for almost every subcritical energy level.
\newblock {\em J. Eur. Math. Soc. (JEMS)}, 19(2):551--579, 2017.

\bibitem{AMP2}
Alberto Abbondandolo, Leonardo Macarini, and Gabriel~P. Paternain.
\newblock On the existence of three closed magnetic geodesics for subcritical energies.
\newblock {\em Comment. Math. Helv.}, 90(1):155--193, 2015.

\bibitem{AbboSchwarz}
Alberto Abbondandolo and Matthias Schwarz.
\newblock A smooth pseudo-gradient for the {Lagrangian} action functional.
\newblock {\em Adv. Nonlinear Stud.}, 9(4):597--623, 2009.

\bibitem{arnold}
V.~I. Arnold.
\newblock {\em Mathematical methods of classical mechanics. {Translated} by {K}. {Vogtman} and {A}. {Weinstein}}, volume~60 of {\em Grad. Texts Math.}
\newblock Springer, Cham, 1978.

\bibitem{lus-fet}
Luca Asselle and Gabriele Benedetti.
\newblock The {Lusternik}-{Fet} theorem for autonomous {Tonelli} {Hamiltonian} systems on twisted cotangent bundles.
\newblock {\em J. Topol. Anal.}, 8(3):545--570, 2016.

\bibitem{assellemazzucchelli}
Luca Asselle and Marco Mazzucchelli.
\newblock On {Tonelli} periodic orbits with low energy on surfaces.
\newblock {\em Trans. Am. Math. Soc.}, 371(5):3001--3048, 2019.

\bibitem{assenza}
Valerio Assenza.
\newblock Magnetic curvature and existence of a closed magnetic geodesic on low energy levels.
\newblock {\em Int. Math. Res. Not.}, 2024(21):13586--13610, 2024.

\bibitem{Assenza3}
Valerio Assenza, Gabriele Benedetti, and Leonardo Macarini.
\newblock Existence and localization of closed magnetic geodesic with low energy.
\newblock (to appear).

\bibitem{assenza2}
Valerio Assenza, James Marshall~Reber, and Ivo Terek.
\newblock Magnetic flatness and {E}. {Hopf}'s theorem for magnetic systems.
\newblock {\em Commun. Math. Phys.}, 406(2):20, 2025.
\newblock Id/No 24.

\bibitem{BahriTaimanov}
Abbas Bahri and Iskander~A. Taimanov.
\newblock Periodic orbits in magnetic fields and {Ricci} curvature of {Lagrangian} systems.
\newblock {\em Trans. Am. Math. Soc.}, 350(7):2697--2717, 1998.

\bibitem{Benci}
V.~Benci.
\newblock Closed geodesics for the {Jacobi} metric and periodic solutions of prescribed energy of natural hamiltonian systems.
\newblock {\em Ann. Inst. Henri Poincar{\'e}, Anal. Non Lin{\'e}aire}, 1:401--412, 1984.

\bibitem{benedettiLectures}
Gabriele Benedetti.
\newblock On closed orbits for twisted autonomous {Tonelli} {Lagrangian} flows.
\newblock {\em Publ. Mat. Urug.}, 16:41--79, 2016.

\bibitem{Boumal}
Nicolas Boumal.
\newblock {\em An Introduction to Optimization on Smooth Manifolds}.
\newblock Cambridge University Press, 2023.

\bibitem{xin}
Da~Rong Cheng and Xin Zhou.
\newblock Existence of constant mean curvature 2-spheres in {Riemannian} 3-spheres.
\newblock {\em Commun. Pure Appl. Math.}, 76(11):3374--3436, 2023.

\bibitem{Contreras1}
G.~Contreras, R.~Iturriaga, G.~P. Paternain, and M.~Paternain.
\newblock The {Palais}-{Smale} condition and {Ma{\~n}{\'e}}'s critical values.
\newblock {\em Ann. Henri Poincar{\'e}}, 1(4):655--684, 2000.

\bibitem{Contreras2}
Gonzalo Contreras.
\newblock The {Palais}-{Smale} condition on contact type energy levels for convex {Lagrangian} systems.
\newblock {\em Calc. Var. Partial Differ. Equ.}, 27(3):321--395, 2006.

\bibitem{contrerasmacarinipaternain}
Gonzalo Contreras, Leonardo Macarini, and Gabriel~P. Paternain.
\newblock Periodic orbits for exact magnetic flows on surfaces.
\newblock {\em Int. Math. Res. Not.}, 2004(8):361--387, 2004.

\bibitem{ginzburggurel}
Viktor~L. Ginzburg and Ba{\c{s}}ak~Z. G{\"u}rel.
\newblock Periodic orbits of twisted geodesic flows and the {Weinstein}-{Moser} theorem.
\newblock {\em Comment. Math. Helv.}, 84(4):865--907, 2009.

\bibitem{Lecturesonclosedgeodesics}
Wilhelm Klingenberg.
\newblock {\em Lectures on closed geodesics}, volume 230 of {\em Grundlehren Math. Wiss.}
\newblock Springer, Cham, 1978.

\bibitem{Merry}
Will~J. Merry.
\newblock Closed orbits of a charge in a weakly exact magnetic field.
\newblock {\em Pac. J. Math.}, 247(1):189--212, 2010.

\bibitem{Struwe}
Michael Struwe.
\newblock Existence of periodic solutions of {Hamiltonian} systems on almost every energy surface.
\newblock {\em Bol. Soc. Bras. Mat., Nova S{\'e}r.}, 20(2):49--58, 1990.

\bibitem{Taimanov}
I.~A. Ta{\u{\i}}manov.
\newblock Closed extremals on two-dimensional manifolds.
\newblock {\em Russ. Math. Surv.}, 47(2):1, 1992.

\bibitem{usher}
Michael Usher.
\newblock Floer homology in disk bundles and symplectically twisted geodesic flows.
\newblock {\em J. Mod. Dyn.}, 3(1):61--101, 2009.

\bibitem{Wojtkowski}
Maciej~P. Wojtkowski.
\newblock Magnetic flows and {Gaussian} thermostats on manifolds of negative curvature.
\newblock {\em Fundam. Math.}, 163(2):177--191, 2000.

\end{thebibliography}

\end{document}